\documentclass{article}
\usepackage{latexsym,amsmath,amssymb,amscd}
\usepackage[mathscr]{eucal}
\usepackage{makeidx}
\usepackage[dvips]{graphicx}
\usepackage{graphpap}

\newtheorem{theorem}{Theorem}

\newtheorem{proposition}[theorem]{Proposition}

\newenvironment{proof}[1][Proof]{\noindent\textbf{#1.} }{\ \rule{0.5em}{0.5em}}
\input{tcilatex}
\begin{document}

\title{{\Large \textbf{Type II\ hidden symmetries for the homogeneous heat
equation in some general classes of Riemannian spaces}}}
\author{Michael Tsamparlis\thanks{%
Email: mtsampa@phys.uoa.gr} \ and Andronikos Paliathanasis\thanks{%
Email: anpaliat@phys.uoa.gr} \\
{\small \textit{Faculty of Physics, Department of Astrophysics - Astronomy -
Mechanics,}}\\
{\small \textit{\ University of Athens, Panepistemiopolis, Athens 157 83,
GREECE}}}
\date{}
\maketitle

\begin{abstract}
We study the reduction of the heat equation in Riemannian spaces which admit
a gradient Killing vector, a gradient homothetic vector and in Petrov Type
D,N,II and Type III space-times. In each reduction we identify the source of
the Type II hidden symmetries. More specifically we find that a) If we
reduce the heat equation by the symmetries generated by the gradient KV the
reduced equation is a linear heat equation in the nondecomposable space. b)
If we reduce the heat equation via the symmetries generated by the gradient
HV the reduced equation is a Laplace equation for an appropriate metric. In
this case the Type II hidden symmetries are generated from the proper CKVs.
c) In the Petrov spacetimes the reduction of the heat equation by the
symmetry generated from the nongradient HV gives PDEs which inherit the Lie
symmetries hence no Type II hidden symmetries appear. We apply the general
results to cases in which the initial metric is specified. We consider the
case that the irreducible part of the decomposed space is a space of
constant nonvanishing curvature and the case of the spatially flat
Friedmann-Robertson-Walker space time used in Cosmology. In each case we
give explicitly the Type II hidden symmetries provided they exist.
\end{abstract}

Keywords: Lie symmetries, Type II hidden symmetries, Heat Equation.

PACS - numbers: 02.20.Sv, 02.30.Jr, 02.40.Ky

\section{Introduction}

Lie point symmetries assist us in the simplification of differential
equations (DE) by means of reduction. The reduction is different for
ordinary differential equations (ODEs) and partial differential equations
(PDEs). In the case of ODEs the use of a Lie point \ symmetry reduces the
order of ODE by one while in the case of PDEs the reduction by a Lie point
symmetry reduces by one the number of independent and dependent variables,
but not the order of the PDE. A common characteristic in the reduction of
both cases is that the Lie point symmetry which is used for the reduction is
not admitted as such by the reduced DE, it is "lost".

It has been found that the reduced equation is possible to admit more Lie
point symmetries than the ones of the original equation. These new Lie point
symmetries have been termed Type II hidden symmetries. Also \ if one works
in the reverse way and either increase the order of an ODE\ or increase the
number of independent and dependent variables of a PDE then it is possible
that the new (the `augmented') DE admits new point symmetries not admitted
by the original DE. This type of Lie point symmetries are called Type I\
hidden symmetries.

The Type I\ and Type II hidden symmetries have been discussed extensively
during the recent years by B Abraham - Shrauner, K S Govinder, P G\ L Leach
and others (see e.g. \cite{Abraham Guo 1994,Abraham Govinder Leach
1995,Leach Govinder Abraham 1999,Abraham Govinder 2008,ASG,Abraham Govinder
Arrige 2006,Abraham 2007,Moitsheko (2004)}). In the following we shall
consider mainly the Type II\ hidden symmetries as they are the ones which
may be used to reduce the reduced DE\ further.

The origin of Type II hidden symmetries is different for the ODEs and the
PDEs although recently it has been shown \cite{Leach Govinder Andriopoulos
2012} that they are nearly the same. For the case of ODEs the inheritance or
not of a Lie point symmetry, the $X_{2}$ say, by the reduced ODE\ depends on
the commutator of that symmetry with the symmetry used for the reduction,
the $X_{1}$ \ say. For example if only two Lie point symmetries $X_{1},X_{2}$
are admitted by the original equation and the commutator $%
[X_{1},X_{2}]=cX_{2}$ where $c$ may be zero,\ then reduction by $X_{1}$
results in $X_{2}$ being a nonlocal symmetry for the reduced ODE while
reduction by $X_{2}$ results in $X_{1}$ being an inherited Lie symmetry of
the reduced ODE. In the reduction by $X_{1}$ the symmetry $X_{2}$ is a Type
I\ hidden symmetry of the original equation relative to the reduced
equation. In the case of more than two Lie point symmetries the situation is
the same if the Lie bracket gives a third Lie point symmetry, the $X_{3}$
say. Then the point like nature of a symmetry is preserved only if reduction
is performed using the normal subgroup \ and $X_{3}$ has a certain
expression \cite{Govinder2001}.

The above scenario is transferred to PDEs as follows. The reduced PDE loses
the symmetry used to reduce the number of variables and it may lose other
Lie point symmetries depending on the structure of the associated Lie
algebra, that is if the admitted subgroup is normal or not \cite%
{Govinder2001}. Again if $X_{1},X_{2}$ are Lie point symmetries of the
original PDE with commutator $[X_{1},X_{2}]=cX_{2}$ where $c$ may be zero,\
then reduction by $X_{2}$ results in $X_{1}$ being a point symmetry of the
reduced PDE while reduction by $X_{1}$ results to an expression which has no
relevance for the PDE \cite{Govinder2001}.

In addition to that scenario, Abraham - Shrauner and Govinder have proposed
a new potential source for the Type II hidden symmetries \cite{ASG,Govinder
Abraham 2009} based on the observation that different PDEs with the same
variables which admit different Lie point symmetry algebras may reduce to
the same target PDE. Based on that observation they propose that the target
PDE inherits Lie point symmetries from all reduced PDEs, which explains why
some of the new symmetries are not admitted by the specific PDE\ used for
the reduction. In this context arises the problem of identifying the set of
all PDEs which lead to the same reduced PDE after reduction by a Lie point
symmetry. In a recent paper \cite{Leach Govinder Andriopoulos 2012} it has
been shown that this is also the case with the ODEs. That is, it is shown
that different differential equations which can be reduced to the same
equation provide point sources for each of the Lie point symmetries of the
reduced equation even though any particular of the higher order equations
may not provide the full complement of Lie point symmetries. Therefore
concerning the ODEs the Lie point symmetries of the reduced equation can be
viewed as having two sources. Firstly the point and nonlocal symmetries of 
\emph{a given} higher order equation and secondly the point symmetries of 
\emph{a variety} of higher order ODEs. Finally in a newer paper \cite%
{Govinder Abraham 2009} it has been shown by a counter example that Type II\
hidden symmetries for PDEs can have a nonpoint origin, i.e. they arise from
contact symmetries or even nonlocal symmetries of the original equation.
Other approaches may be found in \cite{Gandarias 2008,Govinder Abraham2008}.

In the present paper we study the reduction and the consequent existence of
Type II\ hidden symmetries \ of the heat equation in certain classes of
Riemannian spaces.

In a general Riemannian space the heat equation has three Lie point
symmetries which give trivial reduced forms. This implies that if we want to
find `sound' reductions of the heat equation we have to consider Riemannian
spaces which admit some type of symmetry(ies) of the metric (these
symmetries are not Lie symmetries and are called collineations). Indeed it
has been shown \cite{PaliaMT JGP PDFs 2012,Bozkov}, that the Lie symmetries
of the heat and the Poisson equation in a Riemannian space are generated
from the elements of the homothetic algebra and the conformal algebra of the
space respectively. Therefore one expects that in spaces with a nonvoid
homothetic algebra there will be Lie symmetry vectors which will allow for
the reduction of the heat equation and the possibility of the existence of
Type II\ hidden symmetries.

The structure of the paper is as follows. In section \ref{Riemannian spaces
and collineations} we recall results concerning the Lie point symmetry
conditions for a general type of second order PDE. For the convenience of
the reader we also recall two theorems which relate the Lie point symmetries
of the heat equation with the homothetic algebra of the space. The next
sections contain the new results. In section \ref{The heat equation in an
1+n decomposable space} we consider a decomposable space - that is a
Riemannian space which admits a gradient Killing vector (KV). In section \ref%
{The heat equation in a (n+1) space} we do the same for a space which admits
a gradient Homothetic vector (HV) and show that reduction by the Lie
symmetries due to this vector also give rise to Type II\ hidden symmetries.
In section \ref{Applications} we consider the special cases of \ the
previous section, that is, a decomposable space whose nondecomposable part
is a maximally symmetric space of non vanishing curvature and the spatially
flat Friedmann Robertson Walker (FRW) space time used in Cosmology.

In section \ref{The Heat equation in spaces which admit a nongradient HV} we
consider the algebraically special vacuum solutions of Einstein's equations
known as Petrov type N, II, III and D. These space-times belong to
space-times which admit a nongradient HV acting simply transitive \cite%
{Steele 1991 (b)}. In these space-times we consider the reduction of the
heat equation using the Lie point symmetries resulting from the homothetic
vector and look for Type II hidden symmetries. Finally in section \ref%
{Conclusion} we draw our conclusions.

\section{Lie point symmetries and collineations in Riemannian spaces}

\label{Riemannian spaces and collineations}

In a general Riemannian space with metric $g_{ij}$ the heat conduction
equation with flux is%
\begin{equation}
\Delta u-u_{t}=q  \label{WH.0}
\end{equation}%
where $\Delta $ is the Laplace operator $\Delta =\frac{1}{\sqrt{g}}\frac{%
\partial }{\partial x^{i}}\left( \sqrt{g}g^{ij}\frac{\partial }{\partial
x^{j}}\right) $ and $q=q(t,x,u)$. Equation (\ref{WH.0}) can also be written 
\begin{equation}
g^{ij}u_{ij}-\Gamma ^{i}u_{i}-u_{t}=q  \label{WH.001}
\end{equation}%
where $\Gamma ^{i}=\Gamma _{jk}^{i}g^{jk}$ and $\Gamma _{jk}^{i}$ are the
Christofell Symbols of the metric $g_{ij}$.

Equation (\ref{WH.001}) admits the Lie point symmetries 
\begin{equation}
X_{t}=\partial _{t}~\ ,~X_{u}=u\partial _{u}~,~X_{b}=b\left( t,x\right)
\partial _{u}  \label{WH.001a}
\end{equation}%
where $b\left( t,x\right) $ is a solution of the heat equation. These
symmetries are too general to provide sound reductions and consequently
reduced PDEs which can give Type II\ hidden symmetries. However it has been
shown \cite{PaliaMT JGP PDFs 2012,Bozkov} that there is a close relation
between the Lie point symmetries of the heat equation (and its special form
the Poisson equation) with the collineations of the metric. Specifically it
has been shown that the Lie point symmetries of the heat equation are
generated from the HV and the KVs of $g_{ij}$ whereas the Lie point
symmetries of the Poisson equation form the conformal algebra of the metric.
This implies that if we want to have new Lie point symmetries which will
allow for sound reductions of the heat equation eqn (\ref{WH.001}) we have
to restrict our considerations to spaces which admit a homothetic algebra.
Our intention is to keep the discussion as general as possible therefore we
consider spaces in which the metric $g_{ij}$ can be written in generic form.
The spaces we shall consider are:

a. Spaces which admit a gradient KV.

b. Spaces which admit a gradient HV.

The generic form of the metric for these types of spaces has as follows:

\begin{enumerate}
\item If an $1+n-$dimensional Riemannian space admits a gradient KV, the $%
S^{i}=\partial _{x}~\left( S=x\right) $ say, then the space is decomposable
along $\partial _{x}$ and the metric is written as (see e.g. \cite{TNA}) 
\begin{equation*}
ds^{2}=dx^{2}+h_{AB}y^{A}y^{B}~,~h_{AB}=h_{AB}\left( y^{C}\right)
\end{equation*}

\item If an $n-$dimensional Riemannian space admits a gradient HV, the $%
H^{i}=r\partial _{r}~\left( H=\frac{1}{2}r^{2}\right) ,~\psi _{H}=1$ say,
then the metric can be written in the generic form \cite{Tupper}%
\begin{equation*}
ds^{2}=dr^{2}+r^{2}h_{AB}dy^{A}dy^{B}~,~~h_{AB}=h_{AB}\left( y^{C}\right)
\end{equation*}%
where in both cases $h_{AB}$ is the metric of the $n-$ dimensional subspace
with coordinates $\{x^{A}\}$ and the indices $A,B,C=1,...,n$ $.$

The Riemannian spaces which admit nongradient proper HV do not have a
generic form for their metric. However the \ spaces on which the HV\ acts
simply transitively are a few and are given together with their homothetic
algebra in \cite{Steele 1991 (b)}. A\ special class of these spaces are the
algebraically special vacuum space-times known as Petrov types N, II , III,
D whose metric is:

\begin{enumerate}
\item Petrov Type N \qquad 
\begin{equation*}
ds^{2}=dx^{2}+dy^{2}+2d\rho dv-2\ln \left( x^{2}+y^{2}\right) d\rho ^{2}
\end{equation*}

\item Petrov Type D%
\begin{equation*}
ds^{2}=-dx^{2}+x^{-\frac{2}{3}}dy^{2}-x^{\frac{4}{3}}\left( d\rho
^{2}+dz^{2}\right)
\end{equation*}

\item Petrov Type II%
\begin{equation*}
ds^{2}=\rho ^{-\frac{1}{2}}\left( d\rho ^{2}+dz^{2}\right) -2\rho dxdy+\rho
\ln \rho ~dy^{2}
\end{equation*}

\item Petrov Type III%
\begin{equation*}
ds^{2}=2d\rho dv+\frac{3}{2}xd\rho ^{2}+\frac{v^{2}}{x^{3}}\left(
dx^{2}+dy^{2}\right)
\end{equation*}
\end{enumerate}
\end{enumerate}

In what follows all spaces are of dimension $n\succeq 2$. The case $n=1$
although relatively trivial for our approach in general it is not so and has
been studied for example in \cite{Clarkson Mansfield 1993,Ivanova 2008}.

\subsection{Collineations}

For the convenience of the reader and the completeness of the paper we
present some basic results concerning the conformal algebra. In the
following $L_{\xi }$ denotes Lie derivative with respect to the vector field 
$\xi ^{i}$.

A vector field $\xi ^{i}$ is Conformal killing vector (CKV) of a metric $%
g_{ij}$ if $L_{\xi }g_{ij}=2\psi g_{ij}$. If $\psi =0$ then $\xi ^{i}$ is a
KV and if $\psi =1$, $\xi ^{i}$ is a HV. Two metrics $g_{ij},~\bar{g}_{ij}$
are conformally related if $\bar{g}_{ij}=N^{2}g_{ij}$ where the function ~$%
N^{2}$ is the conformal factor.

If $\xi ^{i}$ is a CKV\ of the metric $\bar{g}_{ij}$ so that $L_{\xi }\bar{g}%
_{ij}=2\bar{\psi}\bar{g}_{ij}$ then $\xi ^{i}$ is also a CKV of the metric $%
g_{ij}$, that is $L_{\xi }g_{ij}=2\psi g_{ij},$ where the conformal factor 
\begin{equation*}
\psi =\bar{\psi}N^{2}-NN_{,i}\xi ^{i}.
\end{equation*}%
This means that two conformally related metrics have the same conformal
algebra but with different conformal factors hence subalgebras. This is an
important observation which shall be useful in the following sections.

\subsection{Lie point symmetry conditions}

\label{Lie point symmetry conditions}

In \cite{PaliaMT JGP PDFs 2012} it has been shown that the Lie point
symmetry conditions for the PDE of the form%
\begin{equation}
A^{ij}u_{ij}-B^{i}(x,u)u_{i}-f(x,u)=0  \label{GPE.41}
\end{equation}%
are as follows: 
\begin{equation}
A^{ij}(a_{ij}u+b_{ij})-(a_{i}u+b_{i})B^{i}-\xi
^{k}f_{,k}-auf_{,u}-bf_{,u}+\lambda f=0  \label{GPE.42}
\end{equation}%
\begin{equation}
A^{ij}\xi _{,ij}^{k}-2A^{ik}a_{,i}+aB^{k}+auB_{,u}^{k}-\xi
_{,i}^{k}B^{i}+\xi ^{i}B_{,i}^{k}-\lambda B^{k}+bB_{,u}^{k}=0  \label{GPE.43}
\end{equation}%
\begin{equation}
L_{\xi ^{i}\partial _{i}}A^{ij}=(\lambda -a)A^{ij}-\eta A^{ij}{}_{,u}
\label{GPE.44}
\end{equation}%
\begin{align}
\eta & =a(x^{i})u+b(x^{i})  \label{GPE.45} \\
\xi ^{k}& =\xi ^{k}(x^{i})  \label{GPE.46}
\end{align}%
where the generator of the Lie point symmetry is $\mathbf{X}=\xi ^{i}\left(
x^{i}\right) \partial _{i}+\left( a\left( x^{k}\right) u+b\left(
x^{i}\right) \right) \partial _{u}$ and $b(x^{i})$ is a solution of the PDE.

An immediate conclusion is that if $A_{,u}^{ij}=0$ then from (\ref{GPE.44})
follows that the Lie symmetries are generated from the CKVs of the `metric' $%
A^{ij}.$

In the same paper \cite{PaliaMT JGP PDFs 2012} it has also been shown the
following result concerning the Lie point symmetries of the inhomogeneous
and the homogeneous heat equation.

\begin{theorem}
\label{The Lie of the heat equation with flux}The Lie point symmetries of
the heat equation with flux i.e. 
\begin{equation}
g^{ij}u_{ij}-\Gamma ^{i}u_{i}-u_{t}=q\left( t,x,u\right)  \label{WH.02A}
\end{equation}%
in an $n$- dimensional Riemannian space with metric $g_{ij}$ are constructed
form the homothetic algebra of the metric as follows:

a. $Y^{i}$ is a HV/KV.\newline
The Lie point symmetry is 
\begin{equation}
\mathbf{X}=\left( 2c_{2}\psi t+c_{1}\right) \partial _{t}+c_{2}Y^{i}\partial
_{i}+\left( a\left( t\right) u+b\left( t,x\right) \right) \partial _{u}
\label{HEF.19}
\end{equation}%
where the functions $a(t),b\left( t,x^{k}\right) ,q\left( t,x^{k},u\right) $
satisfy the constraint equation%
\begin{equation}
-a_{t}u+H\left( b\right) -\left( au+b\right) q_{,u}+aq-\left( 2\psi
c_{2}qt+c_{1}q\right) _{t}-c_{2}q_{,i}Y^{i}=0.  \label{HEF.20}
\end{equation}

b. $Y^{i}=S^{,i}$ is a gradient HV/KV.\newline
The Lie point symmetry is 
\begin{equation}
\mathbf{X}=\left( 2\psi \int Tdt\right) \partial _{t}+TS^{,i}\partial
_{i}-\left( \left( \frac{1}{2}T_{,t}S-F\left( t\right) \right) u-b\left(
t,x\right) \right) \partial _{u}  \label{HEF.21}
\end{equation}%
where the functions $F(t),T(t),b\left( t,x^{k}\right) ,q\left(
t,x^{k},u\right) $ satisfy the constraint equation%
\begin{align}
0& =\left( -\frac{1}{2}T_{,t}\psi +\frac{1}{2}T_{,tt}S-F_{,t}\right)
u+H\left( b\right) +  \notag \\
& -\left( \left( -\frac{1}{2}T_{,t}S+F\right) u+b\right) q_{,u}+\left( -%
\frac{1}{2}T_{,t}S+F\right) q-\left( 2\psi q\int Tdt\right)
_{t}-Tq_{,i}S^{,i}.  \label{HEF.22}
\end{align}
\end{theorem}

In the special case in which the heat flux vanishes i.e. $q\left(
t,x,u\right) =0,~$i.e.%
\begin{equation}
g^{ij}u_{ij}-\Gamma ^{i}u_{i}-u_{t}=0  \label{WH.02}
\end{equation}
we have the following result.

\begin{theorem}
\label{The Lie of the heat equation}The Lie point symmetries of the
homogeneous heat equation in an $n-$dimensional Riemannian space 
\begin{equation}
g^{ij}u_{ij}-\Gamma ^{i}u_{i}-u_{t}=0  \label{LHEC.01}
\end{equation}%
are constructed from the homothetic algebra of the metric $g_{ij}$ as
follows:

a. If $Y^{i}$ is a HV/KV of the metric $g_{ij},$ the Lie point symmetry is 
\begin{equation}
\mathbf{X}=\left( 2\psi c_{1}t+c_{2}\right) \partial _{t}+c_{1}Y^{i}\partial
_{i}+\left( a_{0}u+b\left( t,x^{i}\right) \right) \partial _{u}
\label{LHEC.03}
\end{equation}%
where $c_{1},c_{2},,a_{0}$ are constants and $b\left( t,x^{i}\right) $ is a
solution of the homogeneous heat equation.

b. If $Y^{i}=S^{,i}$ is a gradient HV/KV of the metric $g_{ij}$ the Lie
point symmetry is%
\begin{equation}
\mathbf{X}=\psi t^{2}\partial _{t}+tS^{i}\partial _{i}-\left( \frac{1}{2}S+%
\frac{1}{2}n\psi t\right) u\partial _{u}+b\left( t,x^{i}\right) \partial _{u}
\label{LHEC.04+}
\end{equation}%
where $c_{3},c_{4},c_{5}$ are constants and $b\left( t,x^{i}\right) $ is a
solution of the homogeneous heat equation.
\end{theorem}

In both cases $\psi =1$ for a HV and $\psi =0$ for a KV.

In the following we shall need the Lie point symmetries of Laplace equation $%
\Delta u=0$~or%
\begin{equation}
g^{ij}u_{ij}-\Gamma ^{i}u_{i}=0  \label{LE.01}
\end{equation}%
In this context we have the following result

\begin{theorem}[\protect\cite{Bozkov}]
\label{Bozhkov}The Lie point symmetries of Laplace equation (\ref{LE.01})
are generated from the CKVs of the metric $g_{ij}$ as follows\newline
\begin{equation}
X=\xi ^{i}\left( x^{k}\right) \partial _{i}+\left[ \left( \frac{2-n}{2}\psi
\left( x^{k}\right) +a_{0}\right) u+b\left( x^{k}\right) \right] \partial
_{u}
\end{equation}%
where $\xi ^{i}\left( x^{k}\right) $ is a CKV with conformal factor $\psi
\left( x^{k}\right) $ and the following conditions hold%
\begin{equation}
_{g}\Delta \psi =0~,~_{g}\Delta b=0
\end{equation}%
that is, both the conformal factors and the function $b$ are solutions of
Laplace equation.
\end{theorem}

In the following sections we apply these theorems in order to reduced the
heat equation by the extra Lie point symmetries admitted in the special
spaces we considered in section \ref{Riemannian spaces and collineations}.

\section{The heat equation in a $1+n$ decomposable space}

\label{The heat equation in an 1+n decomposable space}

Without loss of generality we assume the gradient KV to be the $\partial
_{x} $ so that the metric has the generic form%
\begin{equation}
ds^{2}=dx^{2}+h_{AB}dy^{A}dy^{B}~,~h_{AB}=h_{AB}\left( y^{C}\right)
\label{WH.03}
\end{equation}%
where $h_{AB}$ $A,B,C=1,...,n$ is the metric of the $n-$ dimensional space.
For the metric (\ref{WH.03}) the heat equation (\ref{WH.02}) takes the form%
\begin{equation}
u_{xx}+h^{AB}u_{AB}-\Gamma ^{A}u_{B}-u_{t}=0.  \label{WH.04}
\end{equation}%
Application of Theorem \ref{The Lie of the heat equation} gives that (\ref%
{WH.04}) admits the following \emph{extra }Lie point symmetries generated by
the gradient KV $\partial _{x}:$ 
\begin{equation*}
X_{1}=\partial _{x}~,~X_{2}=t\partial _{x}-\frac{1}{2}xu\partial _{u}
\end{equation*}%
with nonvanishing commutators 
\begin{equation}
\left[ X_{t},X_{2}\right] =X_{1}~,~\left[ X_{2},X_{1}\right] =\frac{1}{2}%
X_{u}.  \label{WH.04.a}
\end{equation}

We reduce (\ref{WH.04}) using the zeroth order invariants of the extra Lie
point symmetries $X_{1},X_{2}$.

\subsection{Reduction by $X_{1}$}

\label{gradientKV1}

The zeroth order invariants of $X_{1}$ are%
\begin{equation*}
\tau =t~,~y^{A}~,~w=u.
\end{equation*}%
Taking these invariants as new coordinates eqn (\ref{WH.04}) reduces to 
\begin{equation}
_{h}\Delta w-w_{t}=0  \label{WH.05}
\end{equation}%
where $_{h}\Delta $ is the Laplace operator in the~$n-$dimensional space
with metric $h_{AB}:$ 
\begin{equation}
_{h}\Delta w=h^{AB}w_{AB}-\Gamma ^{A}w_{B}.  \label{WH.05.a}
\end{equation}%
Equation (\ref{WH.05}) is the homogeneous heat eqn (\ref{WH.02}) in the $n-$%
\ dimensional space with metric $h_{AB}$. \ According to the Theorem \ref%
{The Lie of the heat equation} the Lie point symmetries of this equation are
the homothetic algebra of $h_{AB}$.{\LARGE \ }It is easy to show that the
homothetic algebra of the $n$ and the $1+n$ metrics are related as follows 
\cite{TNA}:

a. The KVs of the $n-$ metric are identical with those of the $1+n$ metric.

b. The $1+n$ metric admits a HV\ if the $n$ metric admits one and if $%
_{n}H^{A}$ is the HV of the $n$ - metric then the HV\ of the $1+n$ metric is
given by the expression 
\begin{equation}
_{1+n}H^{\mu }=x\delta _{x}^{\mu }+_{n}H^{A}\delta _{A}^{\mu }\qquad \quad
\mu =x,1,...,n.  \label{WH.05.b}
\end{equation}

The above imply that equation (\ref{WH.05}) inherits all symmetries which
are generated from the KVs/HV of the $n-$metric $h_{AB}.$ Hence we do not
have Type II symmetries in this reduction.

\subsection{Reduction by $X_{2}$}

\label{gradientKV2}

The zeroth order invariants of $X_{2}$ are%
\begin{equation*}
\tau =t~,~y^{A}~,~w=ue^{\frac{x^{2}}{4t}}.
\end{equation*}%
Taking these invariants as new coordinates eqn (\ref{WH.04}) reduces to%
\begin{equation}
h^{AB}w_{AB}-\Gamma ^{A}w_{B}-w_{\tau }-\frac{1}{2\tau }w=0  \label{WH.06}
\end{equation}%
or%
\begin{equation*}
_{h}\Delta w-w_{\tau }=\frac{1}{2\tau }w.
\end{equation*}%
This is the nonhomogeneous heat equation with flux $q\left( \tau
,y^{A},w\right) =\frac{1}{2\tau }w$. Application of Theorem \ref{The Lie of
the heat equation with flux} gives the following result.

\begin{proposition}
\label{Cor1}\footnote{%
The proof is given in Appendix \ref{Appendix}.} The Lie point symmetries of
the heat equation (\ref{WH.06}) in an $n-$dimensional Riemannian space with
metric $h_{AB}$ are constructed form the homothetic algebra of the metric as
follows:\newline

a. $Y^{i}$ is a HV/KV.\newline
The Lie point symmetry is 
\begin{equation}
X=\left( 2c_{2}\psi \tau +c_{1}\right) \partial _{\tau }+c_{2}Y^{i}\partial
_{i}+\left[ \left( -\frac{c_{1}}{2\tau }+a_{0}\right) w+b\left( \tau
,x\right) \right] \partial _{w}
\end{equation}

b. $Y^{i}=S_{J}^{,i}$ is a gradient HV/KV (the index $J$ counts gradient
KVs).\newline
The Lie point symmetry is 
\begin{equation}
X=\left( \psi T_{0}\tau ^{2}\right) \partial _{\tau }+T_{0}\tau
S_{J}^{,i}\partial _{i}-\left( \frac{1}{2}T_{0}S_{J}+T_{0}\psi \tau \right)
w\partial _{w}
\end{equation}%
where $b\left( \tau ,x\right) $ is a solution of the heat equation (\ref%
{WH.06}).
\end{proposition}

We infer that for this reduction we have the Type II hidden symmetry ~$%
\partial _{\tau }-\frac{1}{2t}w\partial _{w}.$ The rest of the point
symmetries are inherited.

\section{The heat equation in a space which admits a gradient HV}

\label{The heat equation in a (n+1) space}The metric of an $n$ dimensional
Riemannian space which admits the gradient HV $r\partial _{r}$ has the
generic form \cite{Tupper}%
\begin{equation}
ds^{2}=dr^{2}+r^{2}h_{AB}dy^{A}dy^{B}.  \label{WH.11a}
\end{equation}%
In this space the heat equation (\ref{WH.001}) is 
\begin{equation}
u_{rr}+\frac{1}{r^{2}}h^{AB}u_{AB}+\frac{\left( n-1\right) }{r}u_{r}-\frac{1%
}{r^{2}}\Gamma ^{A}u_{A}-u_{t}=0  \label{WH.11}
\end{equation}%
where $\Gamma ^{A}=\Gamma _{BC}^{A}h^{BC}$ and $\Gamma _{BC}^{A}$ are the
connection coefficients of the Riemannian metric $h_{AB}$ ($%
A,B,C=1,2,...,n). $ Application of Theorem \ref{The Lie of the heat equation}
gives that (\ref{WH.11}) admits the following \emph{extra }Lie point
symmetries generated by the gradient homothetic vector{\LARGE \ }%
\begin{equation}
~\bar{X}_{1}=2t\partial _{t}+r\partial _{r}~\ ,~~\bar{X}_{2}=t^{2}\partial
_{t}+tr\partial _{r}-\left( \frac{1}{4}r^{2}+\frac{n}{2}t\right) u\partial
_{u}~\,  \label{WH.11b}
\end{equation}%
with nonzero commutators 
\begin{equation*}
\left[ X_{t},\bar{X}_{1}\right] =2X_{t}~~,~~\left[ \bar{X}_{1},\bar{X}_{2}%
\right] =2X_{t}
\end{equation*}%
\begin{equation*}
\left[ X_{t},\bar{X}_{2}\right] =\bar{X}_{1}-\frac{n}{2}X_{u}.
\end{equation*}%
We consider again the reduction of (\ref{WH.11}) using the zeroth order
invariants of these extra Lie point symmetries.

\subsection{Reduction by $\bar{X}_{1}$}

\label{HV1}

The zeroth order invariants of $\bar{X}_{1}$ are $\phi =\frac{r}{\sqrt{t}}%
~,~w=u$,~$y^{A}.$ We choose $w=w\left( \phi ,y^{A}\right) $ as the dependent
variable.

Replacing in (\ref{WH.11}) we find the reduced PDE 
\begin{equation}
w_{\phi \phi }+\frac{1}{\phi ^{2}}h^{AB}w_{AB}+\frac{\left( n-1\right) }{%
\phi }w_{\phi }+\frac{\phi }{2}w_{\phi }-\frac{1}{\phi ^{2}}\Gamma
^{A}w_{A}=0.  \label{WH.12}
\end{equation}
Consider a nonvanishing function $N^{2}\left( \phi \right) $ and divide (\ref%
{WH.12}) with $N^{2}\left( \phi \right) $ \ to get:%
\begin{equation}
\frac{1}{N^{2}}w_{\phi \phi }+\frac{1}{\phi^{2} N^{2}}h^{AB}w_{AB}+\frac{%
\left( n-1\right) }{{\phi}N^{2}}w_{\phi }+\frac{\phi }{2N^{2}}w_{\phi }-%
\frac{1}{\phi ^{2}N^{2}}\Gamma ^{A}w_{A}=0  \label{WH.12C}
\end{equation}
It follows that \ (for $n>2$) equation (\ref{WH.12C})\ can be written as%
\begin{equation}
_{\bar{g}}\Delta w=0  \label{WH.12d}
\end{equation}%
where $_{\bar{g}}\Delta $ $\ $\ is the Laplace operator if \ $N^{2}\left(
\phi \right) =\exp \left( \frac{\phi ^{2}}{2\left( n-2\right) }\right) $ and 
$\bar{g}_{ij}$ is the conformally related metric 
\begin{equation}
d\bar{s}^{2}=\exp \left( \frac{\phi ^{2}}{2\left( n-2\right) }\right) \left(
d\phi ^{2}+\phi ^{2}h_{AB}dy^{A}dy^{B}\right).  \label{WH.12e}
\end{equation}

According to \cite{Bozkov} \ the Lie point symmetries of (\ref{WH.12d}) are
the CKVs of the metric (\ref{WH.12e}) \ whose conformal factor satisfies the
condition $_{\bar{g}}\Delta \psi =0.$ Therefore Type II hidden symmetries
will be generated from the proper CKVs. The existence and the number of
these vectors depends mainly on the $n-$ metric $h_{AB}.$

\subsection{Reduction by $\bar{X}_{2}$}

\label{HV2}

For $\bar{X}_{2}$ the zeroth order invariants invariants are $\phi =\frac{r}{%
t}~,~w=ut^{\frac{n}{2}}e^{\frac{r^{2}}{4t}}$,$~y^{A}$. We choose $w=w\left(
\phi ,y^{A}\right) $ as the dependent variable and we have the reduced
equation%
\begin{equation}
_{g}\Delta w=0  \label{WH.13}
\end{equation}%
where%
\begin{equation}
_{g}\Delta w=w_{\phi \phi }+\frac{\left( n-1\right) }{\phi }w_{\phi }+\frac{1%
}{\phi ^{2}}h^{AB}w_{AB}-\frac{1}{\phi ^{2}}\Gamma ^{A}w_{A}.
\end{equation}%
Equation (\ref{WH.13}) is the Laplace equation in the space $\left( \phi
,y^{A}\right) $ with metric%
\begin{equation}
ds^{2}=d\phi ^{2}+\phi ^{2}h_{AB}dy^{A}dy^{B}.  \label{WH.14}
\end{equation}

The Lie symmetries of Laplace equation (\ref{WH.13}) are given in Theorem %
\ref{Bozhkov}. As in the last case the existence and the number of these
vectors depends mainly on the $n$ metric $h_{AB}.$

We note that both vectors $\bar{X}_{1},\bar{X}_{2}$ \ are generated form the
gradient HV and in both cases the heat equation has been reduced to Laplace
equation. This gives the following

\begin{proposition}
\label{Prop}The reduction of the heat equation (\ref{LHEC.01}) in a space
with metric (\ref{WH.11a}) $\left( n>2\right) ~$by means of the Lie
symmetries generated by the gradient HV leads to Laplace equation $\Delta
u=0 $, where \thinspace $\Delta $ is the Laplace operator for the metric (%
\ref{WH.12e}) if the reduction is done by $\bar{X}_{1}$ and for the metric (%
\ref{WH.14}) if the reduction is done by $\bar{X}_{2}$.
\end{proposition}

\section{Applications}

\label{Applications}

In this section we consider applications of the general results of sections %
\ref{The heat equation in an 1+n decomposable space} and \ref{The heat
equation in a (n+1) space}.

\subsection{The heat equation in a $1+n$ decomposable space where $n$ is a
space of constant curvature}

Consider the $1+n$ decomposable space%
\begin{equation}
ds^{2}=dx^{2}+N^{-2}\left( y^{C}\right) \delta _{AB}y^{A}y^{B}  \label{WH.17}
\end{equation}%
where $N\left( y^{C}\right) =\left( 1+\frac{K}{4}y^{C}y_{C}\right) ,$ that
is, the $n$ space is a space of constant non vanishing ($K\neq 0)$
curvature. The metric (\ref{WH.17}) \ does not admit proper HV. However
admits $\frac{n\left( n-1\right) }{2}+n~$nongradient KVs and $1$ gradient KV
as follows \cite{TNA}%
\begin{eqnarray*}
1~\text{gradient KV}\text{:~} &&\partial _{x} \\
n~\text{nongradient KVs} &\text{:}&\text{\ }K_{V}=\frac{1}{N}\left[ \left(
2N-1\right) \delta _{I}^{i}+\frac{K}{2}Nx_{I}x^{i}\right] \partial _{i} \\
\frac{n\left( n-1\right) }{2}\text{nongradient KVs} &\text{:}&\text{\ }%
X_{IJ}=\delta _{\lbrack I}^{j}\delta _{J]}^{i}\partial _{i}
\end{eqnarray*}%
In a space with metric (\ref{WH.17}) the heat equation takes the form%
\begin{equation}
u_{xx}+N^{2}\left( y^{C}\right) \delta ^{AB}u_{AB}-\frac{N}{2}%
Ky^{A}u_{A}-u_{t}=0  \label{WH.18}
\end{equation}%
which is the homogeneous heat equation. Applying Theorem \ref{The Lie of the
heat equation} we find that equation (\ref{WH.18}) admits the extra Lie
point symmetries 
\begin{equation}
\partial _{x}~,~t\partial _{x}-\frac{1}{2}xu\partial _{x}~,K_{V}~,X_{IJ}.
\end{equation}%
The Lie point symmetries which are generated by the gradient KV are\footnote{%
Here the algebra is the one given in section \ref{The heat equation in an
1+n decomposable space} and a separate algebra is the algebra of the KVs of
the space of constant curvature. More specifically the KVs $K_{V}~,X_{IJ}$
commute with all other symmetries but not between themselves} $\partial
_{x}~,~t\partial _{x}-\frac{1}{2}xu\partial _{x}$.\ 

Reduction of (\ref{WH.18}) by means of the gradient KV $\partial _{x}$
results in the special form of equation (\ref{WH.05}) 
\begin{equation}
\frac{1}{N^{2}\left( y^{C}\right) }\delta ^{AB}u_{AB}-\frac{N}{2}%
Ky^{A}u_{A}-u_{t}=0.  \label{WH.19}
\end{equation}%
This is the homogeneous heat equation in an $n$- dimensional space of
constant curvature. The Lie symmetries of this equation have been determined
in \cite{PaliaMT JGP PDFs 2012} and are inherited symmetries. Hence in this
case we do not have Type II\ hidden symmetries.

Reduction of (\ref{WH.18}) with the Lie point symmetry $~t\partial _{x}-%
\frac{1}{2}xu\partial _{x}$ \ gives that the reduced equation (\ref{WH.06})
is 
\begin{equation}
N^{2}\left( y^{C}\right) \delta ^{AB}w_{AB}-\frac{N}{2}Ky^{A}w_{A}-w_{\tau }=%
\frac{1}{2\tau }w~,~w=ue^{\frac{x^{2}}{4t}}.  \label{WH.20}
\end{equation}%
which is the heat equation with flux. By Proposition \ref{Cor1} the Lie
point symmetries of (\ref{WH.20}) are:%
\begin{equation}
X=c_{1}\partial _{\tau }+\left( K_{V}~+X_{IJ}\right) +\left[ \left( -\frac{%
c_{1}}{2\tau }+a_{0}\right) w+b\left( \tau ,y^{C}\right) \right] \partial
_{w}.  \label{WH.21}
\end{equation}%
where $c_{1},a_{0}$ are constants. From section \ref{gradientKV2} we have
that Type II hidden symmetry is the one defined by the constant $c_{1}$.

\subsection{FRW space-time with a gradient HV}

Consider the spatially flat FRW\ metric widely used in Cosmology%
\begin{equation}
ds^{2}=d\sigma ^{2}-\sigma ^{2}\left( dx^{2}+dy^{2}+dz^{2}\right)
\label{WH.22a}
\end{equation}%
where $\tau $ is the conformal time. This metric admits the gradient HV%
\begin{equation*}
H=\sigma \partial _{\sigma }~~\left( \psi _{H}=1\right)
\end{equation*}%
and six nongradient KVs%
\begin{equation*}
X_{1-3}=\partial _{y^{A}}~~,~~X_{4-6}=y^{B}\partial _{A}-y^{A}\partial _{B}.
\end{equation*}%
where $y^{A}=\left( x,y,z\right) .~$

In this space the heat equation takes the form%
\begin{equation}
u_{\sigma \sigma }-\frac{1}{\sigma ^{2}}\left( u_{xx}+u_{yy}+u_{zz}\right) +%
\frac{3}{\sigma }u_{\sigma }-u_{t}=0.  \label{WH.22}
\end{equation}%
Its Lie point symmetries (\ref{WH.22}) have been determined in \cite{PaliaMT
JGP PDFs 2012} and have as follows:%
\begin{eqnarray*}
&&\partial _{t}~,~u\partial _{u}~,~b\left( \tau ,y^{A}\right) \partial
_{u}~,~X_{1-3}~,~X_{4-6}~, \\
H_{1} &=&2t\partial _{t}+\sigma \partial _{\sigma }~,~H_{2}=t^{2}\partial
_{t}+t\sigma \partial _{\sigma }-\left( \frac{1}{4}\sigma ^{2}+2t\right)
u\partial _{u}~.
\end{eqnarray*}%
The Lie symmetries $H_{1},H_{2}$ are produced by the gradient HV therefore
we use them to reduce (\ref{WH.22}). We note that this case is a special
case of the one we considered in section \ref{The heat equation in a (n+1)
space}\ for $h_{AB}=\delta _{AB}$.

Reduction by $H_{1}$ gives that (\ref{WH.22})~becomes: 
\begin{equation}
w_{\phi \phi }-\frac{1}{\phi ^{2}}\left( w_{xx}+w_{yy}+w_{zz}\right) +\left( 
\frac{3}{\phi }+\frac{\phi }{2}\right) w_{\phi }=0  \label{WH.23}
\end{equation}%
where $\phi =\frac{r}{\sqrt{\sigma }}~,~w=u.$ This is a special form of (\ref%
{WH.12}).

Dividing with $N^{2}\left( \phi \right) =\exp \left( \frac{\phi ^{2}}{4}%
\right) $ we find that (\ref{WH.23}) is written as 
\begin{equation}
_{\bar{g}}\Delta w=0  \label{WH.23a}
\end{equation}%
where the metric $\bar{g}_{ij}$ is the conformally related metric of \ (\ref%
{WH.22a}):%
\begin{equation}
d\bar{s}^{2}=e^{\frac{\phi ^{2}}{4}}\left( d\phi ^{2}-\phi ^{2}\left(
dx^{2}+dy^{2}+dz^{2}\right) \right)  \label{WH.23aa}
\end{equation}

From proposition \ref{Bozhkov} we have that the Lie point symmetries of (\ref%
{WH.23a}) are generated from elements of the conformal algebra of the space
whose conformal factors satisfy the condition $_{\bar{g}}\Delta \psi =0.$
The metric (\ref{WH.23a}) is conformally flat therefore its conformal group
is the same with that of the flat space \cite{TNA,Maartens}, however with
different subgroups. We find that these vectors (i.e. the Lie point
symmetries)\ are the vectors%
\begin{equation}
~X_{1-3}~,~X_{4-6}~,\partial _{t},\text{ }w\partial _{w}~,~b_{0}\left( \phi
,y^{A}\right) \partial _{w}.~  \label{WH.24}
\end{equation}%
We conclude that there are no Type II symmetries for this reduction.

Using reduction by $H_{2}$ we find that (\ref{WH.22})~reduces to :%
\begin{equation}
w_{\phi \phi }-\frac{1}{\phi ^{2}}\left( w_{xx}+w_{yy}+w_{zz}\right) +\frac{3%
}{\phi }w_{\phi }=0  \label{WH.25}
\end{equation}%
where $\phi =\frac{\tau }{\tau }~,~w=ut^{2}e^{\frac{\tau ^{2}}{4t}}.~$\ This
is a special form of (\ref{WH.13}) which is the Laplace equation. In this
case the results of \cite{Bozkov} apply and we infer that the Lie point
symmetries of (\ref{WH.25}) are: \ 
\begin{eqnarray}
&&X_{1-3}~,~X_{4-6}~,~~w\partial _{w}~,~b_{1}\left( \phi ,y^{A}\right)
\partial _{w}  \notag \\
X_{7} &=&\phi \partial _{\phi }~,~X_{8-10}=\phi y^{A}\partial _{\phi }+\ln
\phi \partial _{A}-y^{A}w\partial _{w}.
\end{eqnarray}%
The vector $X_{7}$ is the proper HV of the metric and the vectors $X_{8-10}$
the proper CKVs which are not special CKVs, therefore these vectors are Type
II\ hidden symmetries. A further analysis of (\ref{WH.25}) can be found in 
\cite{Kara}.

\section{The Heat equation in spaces which admit a nongradient HV}

\label{The Heat equation in spaces which admit a nongradient HV}In the
previous sections we considered the reduction of the homogeneous heat
equation in Riemannian spaces which admit a gradient KV or a gradient HV. In
the present section we consider the special class of Petrov type space-times
which admit a nongradient HV which acts simply transitively.

\subsection{Petrov type N space-time}

The metric of the Petrov type N space-time is 
\begin{equation}
ds^{2}=dx^{2}+x^{2}dy^{2}+2d\rho dv+\ln x^{2}d\rho ^{2}  \label{WH.07b}
\end{equation}%
and has the homothetic algebra \cite{Steele 1991 (b)}%
\begin{eqnarray*}
K^{1} &=&\partial _{\rho }~,~K^{2}=\partial _{v}~,~K^{3}=\partial _{y} \\
H &=&x\partial _{x}+\rho \partial _{\rho }+\left( v-2\rho \right) \partial
_{v}~~\left( \psi _{H}=1\right)
\end{eqnarray*}%
where $K^{1-3}$ are KVs and $H$ is a nongradient HV.

The heat equation (\ref{WH.02}) in this space-time is 
\begin{equation}
u_{xx}+\frac{1}{x^{2}}u_{yy}+2u_{\rho v}-2\ln x^{2}~u_{vv}+\frac{1}{x}%
u_{x}-u_{t}=0.  \label{WH.07}
\end{equation}%
Application of Theorem \ref{The Lie of the heat equation} gives that the
extra Lie point symmetries of (\ref{WH.07}) are 
\begin{equation*}
X_{1-3}=K_{1-3}~\ ,~X_{4}=2t\partial _{t}+H
\end{equation*}%
with nonzero commutators 
\begin{equation*}
\left[ X_{t},X_{4}\right] =2X_{t}
\end{equation*}%
\begin{equation*}
\left[ X_{1},X_{4}\right] =X_{1}-2X_{2}~,~\left[ X_{2},X_{4}\right] =X_{2}.~
\end{equation*}

We use $X_{4}$ to reduce the PDE because this is the Lie symmetry generated
by the HV. The zeroth order invariants of $X_{4}$ are\qquad\ 
\begin{equation}
\alpha =\frac{x}{\sqrt{t}}~,~\beta =\frac{\rho }{\sqrt{t}}~,~\gamma =\frac{%
v+\rho \ln \left( t\right) }{\sqrt{t}}~,~\delta =y~,~w=u.  \label{WW.0a}
\end{equation}%
Choosing $\alpha ,\beta ,\gamma ,\delta $ as the independent variables and $%
w=w\left( \alpha ,\beta ,\gamma ,\delta \right) $ as the dependent variable
we find that the reduced PDE is%
\begin{equation}
_{N}\Delta w+\left( \frac{1}{2}\alpha w_{\alpha }+\frac{1}{2}\beta w_{\beta
}+\left( \frac{1}{2}\gamma -\beta \right) w_{\gamma }\right) =0.
\label{WH.08}
\end{equation}%
where $_{N}\Delta $ is the Laplace operator for the metric (\ref{WH.07b}).

Equation (\ref{WH.08}) is of the form (\ref{GPE.41}) with $%
A_{ij}=g_{ij}\left( x^{k}\right) $, $B^{i}\left( x^{k}\right) =\Gamma ^{i}+%
\frac{1}{2}\alpha \delta _{\alpha }^{i}+\frac{1}{2}\beta \delta _{\beta
}^{i}+\left( \frac{1}{2}\gamma -\beta \right) \delta _{\gamma }^{i},~f\left(
x^{k},u\right) =0,~$where $g_{ij}$ is the metric (\ref{WH.07b}). Replacing
in equations (\ref{GPE.42})-(\ref{GPE.46}) we obtain the Lie symmetry
conditions for (\ref{WH.08}). \ Because $A_{ij,u}=0$ it follows from
equation (\ref{GPE.44}) that the Lie point symmetries are generated from the
CKVs of the metric (\ref{WH.07b}). However taking into consideration the
rest of the symmetry conditions we find that the only Lie point symmetry
which remains is the one of the KV\ $X_{3}.$ We conclude that in this
reduction we do not have Type II\ hidden symmetries.

\subsection{Petrov type D}

The metric of the Petrov type D space-time is 
\begin{equation}
ds^{2}=-dx^{2}+x^{-\frac{2}{3}}dy^{2}-x^{\frac{4}{3}}\left( d\rho
^{2}+dz^{2}\right)  \label{WH.09b}
\end{equation}%
with Homothetic algebra%
\begin{eqnarray*}
K^{1} &=&\partial _{\rho }~,~K^{2}=\partial _{z}~,~K^{3}=\partial
_{y}~,~K^{4}=z\partial _{\rho }-\rho \partial _{z} \\
H &=&x\partial _{x}+\frac{4}{3}y\partial _{y}+\frac{z}{3}\partial _{z}+\frac{%
\rho }{3}\partial _{\rho }~~\left( \psi _{H}=1\right)
\end{eqnarray*}%
where $K^{1-4}$ are KVs and $H$ is a nongradient HV.

In this space-time the heat equation (\ref{WH.02}) takes the form: 
\begin{equation}
-u_{xx}+x^{\frac{2}{3}}u_{yy}-x^{-\frac{4}{3}}\left( u_{\rho \rho
}+u_{zz}\right) -\frac{1}{x}u_{x}-u_{t}=0.  \label{WH.09}
\end{equation}%
From\ Theorem\ \ref{The Lie of the heat equation} we have\ that the extra
Lie point symmetries are the vectors 
\begin{equation*}
X_{1-4}=K_{1-4}~,~~X_{5}=2t\partial _{t}+H.
\end{equation*}%
with nonzero commutators:%
\begin{eqnarray*}
\left[ X_{t},X_{5}\right] &=&2X_{t} \\
\left[ X_{1},X_{5}\right] &=&\frac{1}{3}X_{1}~,~\left[ X_{4},X_{1}\right]
=-X_{2} \\
\left[ X_{2},X_{4}\right] &=&X_{1}~,~\left[ X_{2},X_{5}\right] =\frac{1}{3}%
X_{2} \\
\left[ X_{3},X_{5}\right] &=&\frac{4}{3}X_{3}.
\end{eqnarray*}

We use $X_{5}$ to reduce the PDE because this is the Lie symmetry generated
by the HV. The zeroth order invariants of $X_{5}$ are%
\begin{equation*}
\alpha =\frac{x}{t^{\frac{1}{2}}}~,~\beta =\frac{y}{t^{\frac{2}{3}}}%
~,~\gamma =\frac{\rho }{t^{\frac{1}{6}}}~,~\delta =\frac{z}{t^{\frac{1}{6}}}%
~,~w=u.
\end{equation*}%
We choose $\alpha ,\beta ,\gamma ,\delta $ as the independent variables and $%
w=w\left( \alpha ,\beta ,\gamma ,\delta \right) $ as the dependent variable
and we find that the reduced PDE is%
\begin{equation}
_{D}\Delta w+\left( \frac{1}{2}aw_{\alpha }+\frac{2}{3}\beta w_{\beta }+%
\frac{1}{6}\gamma w_{\gamma }+\frac{1}{6}\delta w_{\delta }\right) =0
\label{WH.10a}
\end{equation}%
where~$_{D}\Delta $ is the Laplace operator with metric (\ref{WH.09b}).

Again working with the Lie symmetry conditions (\ref{GPE.42})-(\ref{GPE.46})
we find that equation (\ref{WH.10a}) admits as Lie point symmetry only the
vector $X_{4}$ which is an inherited symmetry. Hence we do not have Type II\
hidden symmetries.\ Obviously the Lie point symmetries $X_{1-4}$\ are Type
I\ hidden symmetries for equation (\ref{WH.10a})\ for the reduction by $%
X_{5}.$

\subsection{Petrov type II}

The metric of the Petrov type II space-time is 
\begin{equation}
ds^{2}=\rho ^{-\frac{1}{2}}\left( d\rho ^{2}+dz^{2}\right) -2\rho dxdy+\rho
\ln \rho ~dy^{2}  \label{WH.II1}
\end{equation}%
with Homothetic algebra%
\begin{eqnarray*}
K^{1} &=&\partial _{x}~,~K^{2}=\partial _{y}~,~K^{3}=\partial _{z}~ \\
H &=&\frac{1}{3}\left( x+2y\right) \partial _{x}+\frac{1}{3}y\partial _{y}+%
\frac{4}{3}z\partial _{z}+\frac{4}{3}\rho \partial _{\rho }~\left( \psi
_{H}=1\right)
\end{eqnarray*}%
where $K^{1-4}$ are KVs and $H$ is a nongradient HV.

In this space-time equation (\ref{WH.02}) takes the form:%
\begin{equation}
\rho ^{\frac{1}{2}}\left( u_{\rho \rho }+u_{zz}\right) -\frac{1}{\rho }%
\varepsilon \ln \rho u_{xx}-\frac{2}{\rho }u_{xy}+\rho ^{-\frac{1}{2}%
}u_{\rho }-u_{t}=0.
\end{equation}%
From Theorem \ref{The Lie of the heat equation} we have that the extra Lie
point symmetries are the vectors 
\begin{equation*}
X_{1-3}=K_{1-3}~,~~X_{4}=2t\partial _{t}+H
\end{equation*}%
with nonzero commutators:%
\begin{equation*}
\left[ X_{t},X_{4}\right] =2X_{t}
\end{equation*}%
\begin{eqnarray*}
\left[ X_{1},X_{4}\right] &=&\frac{1}{3}X_{1}~,~\left[ X_{3},X_{4}\right] =%
\frac{4}{3}X_{3} \\
\left[ X_{2},X_{4}\right] &=&\frac{2}{3}X_{1}+\frac{1}{3}X_{2}.
\end{eqnarray*}%
We use $X_{4}$ to reduce the PDE because this is the Lie symmetry generated
by the HV. The zeroth order invariants of $X_{4}$ are%
\begin{equation*}
\alpha =\frac{\rho }{t^{\frac{2}{3}}}~,~\beta =\frac{z}{t^{\frac{2}{3}}}%
~,~\gamma =\frac{x-\frac{1}{3}y\ln \left( t\right) }{t^{\frac{1}{6}}}%
~,~\delta =\frac{y}{t^{\frac{1}{6}}}~~,~w=u
\end{equation*}%
We choose $\alpha ,\beta ,\gamma ,\delta $ as the independent variables and $%
w=w\left( \alpha ,\beta ,\gamma ,\delta \right) $ as the dependent variable
and we find that the reduced PDE is%
\begin{equation}
_{II}\Delta w+\frac{2}{3}aw_{\alpha }+\frac{2}{3}\beta w_{\beta }+\left( 
\frac{1}{3}\delta -\frac{1}{6}\gamma \right) w_{\gamma }+\frac{1}{6}\delta
w_{\delta }=0  \label{WH.II}
\end{equation}%
where $_{II}\Delta $ is the Laplace operator for metric (\ref{WH.II1}).

From the Lie symmetry conditions follows that (\ref{WH.II}) does not admit
any Lie point symmetries. Hence we do not have Type II\ hidden symmetries in
this case.

\subsection{Petrov type III}

The metric of the Petrov type III space-time is 
\begin{equation}
ds^{2}=2d\rho dv+\frac{3}{2}xd\rho ^{2}+\frac{v^{2}}{x^{3}}\left(
dx^{2}+dy^{2}\right)  \label{WH.III}
\end{equation}%
with Homothetic algebra%
\begin{eqnarray*}
K^{1} &=&\partial _{\rho }~,~K^{2}=\partial _{y}~,~K^{3}=v\partial _{v}-\rho
\partial _{\rho }+2x\partial _{x}+2y\partial _{y} \\
H &=&v\partial _{v}+\rho \partial _{\rho }~\left( \psi _{H}=1\right)
\end{eqnarray*}%
where $K^{1-4}$ are KVs and $H$ is a nongradient HV.

In this space-time equation (\ref{WH.02}) takes the form:%
\begin{equation}
-\frac{3}{2}xu_{vv}+2u_{v\rho }+\frac{x^{3}}{v^{2}}\left(
u_{xx}+u_{yy}\right) -3\frac{x}{v}u_{v}+\frac{2}{v}u_{\rho }-u_{t}=0.
\end{equation}%
From\ Theorem\ \ref{The Lie of the heat equation} we have\ that the extra
Lie point symmetries are the vectors 
\begin{equation*}
X_{1-3}=K_{1-3}~,~~X_{4}=2t\partial _{t}+H
\end{equation*}%
with nonzero commutators:%
\begin{equation*}
\left[ X_{t},X_{4}\right] =2X_{t}~,~\left[ X_{2,},X_{3}\right] =2X_{2}
\end{equation*}%
\begin{equation*}
\left[ X_{3},X_{1}\right] =X_{1}~,~~\left[ X_{1},X_{4}\right] =X_{1}~
\end{equation*}%
We use $X_{4}$ for reduction because this is the Lie symmetry generated by
the HV.

The zeroth order invariants of $X_{4}$ are%
\begin{equation*}
\alpha =\frac{v}{\sqrt{t}}~,~\beta =\frac{\rho }{\sqrt{t}}~,~\gamma
=x~,~\delta =y~,~w=u
\end{equation*}%
We choose $w=w\left( \alpha ,\beta ,\gamma ,\delta \right) $ as the
dependent variable and we find that the reduced PDE is%
\begin{equation}
_{III}\Delta w+\frac{\alpha }{2}w_{\alpha }+\frac{\beta }{2}w_{\beta }=0
\label{WH.III1}
\end{equation}%
where $_{III}\Delta $ is the Laplace operator for metric (\ref{WH.III}).

From the Lie symmetry conditions we find that \ (\ref{WH.III1}) admits only
the Lie point symmetries ~$X_{2},X_{3}.$ Therefore\ we do not have type II
hidden symmetries and the symmetries \ $X_{1},X_{4}~$ are Type I\ hidden
symmetries.

\section{Conclusion}

\label{Conclusion}

We have discussed the reduction of the homogeneous heat equation in certain
general classes of Riemannian spaces which admit some type of basic symmetry
and we have determined in each case the Type II\ hidden symmetries. These
spaces are the spaces which admit a gradient KV or a gradient HV and finally
space-times which admit a HV\ which acts simply and transitively.

In general the reduction of the homogeneous heat equation in these spaces
leads to reduced equations in which the second order partial derivatives are
within a Laplace operator, that is, the reduced equations are of the form $%
\Delta u=F\left( t,x^{i},u_{i},u_{t}\right) $ where $F$ is linear on $%
u_{i},~u_{t}.$ This implies that the Lie point symmetries of the reduced
equation will be generated from the elements of the conformal algebra of the
metric which defines the Laplace operator $\Delta $. It is from these
symmetries that the Type II\ hidden symmetries will emerge. Summarizing we
have found the following general geometric results:

\begin{itemize}
\item If we reduce the heat equation (\ref{WH.02}) via the symmetries which
are generated by a gradient KV~$\left( S^{,i}\right) $ the reduced equation
is a heat equation in the nondecomposable space. In this case we have the
Type II hidden symmetry $\partial _{t}-\frac{1}{2t}w\partial _{w}$ provided
we reduce the heat equation with the symmetry~$tS^{,i}-\frac{1}{2}Su\partial
_{u}$.

\item If we reduce the heat equation (\ref{WH.02}) via the symmetries which
are generated by a gradient HV the reduced equation is Laplace equation for
an appropriate metric. In this case the Type II hidden symmetries are
generated from the proper CKVs.

\item In Petrov spacetimes the reduction of the heat equation (\ref{WH.02})
via the symmetry generated from the nongradient HV gives PDEs which inherit
the Lie symmetries, hence no Type II hidden symmetries are admitted.
\end{itemize}

The results we have obtained can be used in many important space-times and
help facilitate the solution of the heat equation in these space-times.
Finally we mention that the results we have obtained have been checked with
the libraries SADE \cite{SADE} and PDEtools \cite{PDEtools} of Maple%
\footnote{%
http://www.maplesoft.com/}.

\appendix

\section{Appendix}

\label{Appendix}

\begin{proof}[Proof of Corollary \protect\ref{Cor1}]
Using Theorem \ref{The Lie of the heat equation with flux} and replacing $q$
we have

For case a) 
\begin{eqnarray}
-a_{\tau }w+H\left( b\right) -\frac{1}{2\tau }\left( aw+b\right) +\frac{a}{%
2\tau }w-\left( \psi c_{2}w+\frac{1}{2\tau }wc_{1}\right) _{\tau } &=&0. \\
-a_{\tau }w+H\left( b\right) -\frac{1}{2\tau }b+\frac{1}{2\tau ^{2}}wc_{1}
&=&0 \\
\left[ -a_{\tau }+\frac{c_{1}}{2\tau ^{2}}\right] w+\left[ H\left( b\right) -%
\frac{1}{2\tau }b\right] &=&0
\end{eqnarray}%
that is%
\begin{equation}
a=-\frac{c_{1}}{2\tau }+a_{0}~,~~H\left( b\right) -\frac{1}{2\tau }b=0
\end{equation}

For case b)%
\begin{equation}
0=\left( -\frac{1}{2}T_{,\tau }\psi +\frac{1}{2}T_{,\tau \tau }S-F_{,\tau
}\right) w-\left( 2\psi q\int Td\tau \right) _{\tau }-Tq_{,i}S^{,i}.  \notag
\end{equation}%
then\qquad 
\begin{equation}
0=\left( -\frac{1}{2}T_{,\tau }\psi +\frac{1}{2}T_{,\tau \tau }S-F_{,\tau
}\right) w+\frac{\psi }{\tau ^{2}}\int Td\tau ~w-\frac{\psi }{\tau }Tw
\end{equation}%
from here we have%
\begin{equation}
T_{,\tau \tau }=0\rightarrow ~T=T_{0}\tau +T_{1}
\end{equation}%
and%
\begin{equation*}
F=-T_{0}\psi \tau .
\end{equation*}
\end{proof}

\end{document}